\documentclass[a4paper,11pt]{amsart}
\usepackage{amsfonts}
\usepackage{amssymb}
\usepackage[utf8]{inputenc}
\usepackage{amsmath}
\usepackage{pdflscape}
\usepackage{graphicx}
\providecommand{\U}[1]{\protect\rule{.1in}{.1in}}
\usepackage[colorlinks=true, linkcolor=red, citecolor=blue]{hyperref}
\usepackage[]{epsfig}
\usepackage[]{pstricks}
\usepackage{tikz}
\providecommand{\U}[1]{\protect\rule{.1in}{.1in}}
\newtheorem{theorem}{Theorem}[section]
\newtheorem{proposition}[theorem]{Proposition}

\newtheorem{corollary}[theorem]{Corollary}
\newtheorem{remark}{Remark}
\theoremstyle{remark}

\newcommand{\remove}[1]{ }

\def\be{\begin{equation}}
\def\ee{\end{equation}}
\def\ba{\begin{eqnarray}}
\def\ea{\end{eqnarray}}

\def\vf{\varphi}

\numberwithin{equation}{section}

\makeatletter
\@namedef{subjclassname@2020}{\textup{2020} Mathematics Subject Classification}
\makeatother

\begin{document}

	\pagenumbering{arabic}	
\title[Stability of KdV equation]{Weak damping for the Korteweg-de Vries equation}
\author[Capistrano--Filho]{Roberto de A. Capistrano--Filho*}\thanks{*This work is dedicated to my daughter Helena}
\address{Departamento de Matem\'atica,  Universidade Federal de Pernambuco (UFPE), 50740-545, Recife (PE), Brazil.}
\email{roberto.capistranofilho@ufpe.br}

\begin{abstract}
For more than 20 years, the Korteweg-de Vries equation has been intensively explored from the mathematical point of view. Regarding control theory, when adding an internal force term in this equation, it is well known that the Korteweg-de Vries equation is exponentially stable in a bounded domain. In this work, we propose a weak forcing mechanism, with a lower cost than that already existing in the literature,  to achieve the result of the global exponential stability to the Korteweg-de Vries equation.
\end{abstract}
\maketitle

\section{Introduction\label{Sec0}}

\subsection{Historical review}
In 1834 John Scott Russell, a Scottish naval engineer, was observing the Union Canal in Scotland when he unexpectedly witnessed a very special physical phenomenon which he called a wave of translation \cite{Russel1844}. He saw a particular wave traveling through this channel without losing its shape or velocity, and was so captivated by this event that he focused his attention on these waves for several years, not only built water wave tanks at his home conducting  practical and theoretical research into these types of waves, but also challenged the mathematical community to prove theoretically the existence of his solitary waves and to give an a priori demonstration a posteriori.

A number of researchers took up Russell’s challenge. Boussinesq  was  the  first  to  explain  the  existence  of  Scott  Russell’s  solitary wave  mathematically.  He  employed  a  variety  of  asymptotically  equivalent equations to describe water waves in the small-amplitude, long-wave regime. In fact, several  works  presented  to  the  Paris  Academy  of  Sciences  in  1871  and 1872, Boussinesq addressed the problem of the persistence of solitary waves of permanent form on a fluid interface \cite{Boussinesq,Boussinesq1,Boussinesq2,Boussinesq3}.
It is important to mention that in 1876, the English physicist Lord Rayleigh obtained a different result \cite{Rayleigh}.

After Boussinesq theory,  the Dutch mathematicians D. J. Korteweg and his student G. de Vries  \cite{Korteweg} derived 
a nonlinear partial differential equation in 1895 that  possesses  a solution describing the phenomenon discovered by Russell, 
\begin{equation}\label{kdv}\frac{\partial\eta}{\partial{t}}=\frac{3}{2}\sqrt{\frac{g}{l}}\frac{\partial}{\partial{x}}\left(\frac{1}{2}\eta^2+\frac{3}{2}\alpha\eta+\frac{1}{3}\beta\frac{\partial^2\eta}{\partial{x^2}}\right),
\end{equation}
in which $\eta$ is the surface elevation above the equilibrium level, $l$ is an arbitrary constant related to the motion of the liquid, $g$ is the gravitational constant, and $\beta=\frac{l^3}{3}-\frac{Tl}{\rho g}$ with surface capillary tension $T$ and density $\rho$. 
The equation (\ref{kdv})  is called   the Korteweg-de Vries equation in the literature, often abbreviated as the KdV equation,   although  it had appeared explicitly in 
\cite{Boussinesq3}, as equation (283bis) in  a  footnote  on  page  360\footnote{The interested readers are referred to \cite{jager2006, pego1998} for history and origins  of the Korteweg-de Vries equation.}.

Eliminating the physical constants by using the following change of variables
$$t\to\frac{1}{2}\sqrt{\frac{g}{l\beta}}t, \quad x\to-\frac{x}{\beta}, \quad u\to-\left(\frac{1}{2}\eta+\frac{1}{3}\alpha\right)$$
one obtains the standard Korteweg-de Vries (KdV) equation 
\begin{equation}\label{kdv1}
u_t + 6uu_x + u_{xxx}= 0
\end{equation}
which is now commonly accepted as a mathematical model for the unidirectional propagation of small-amplitude long waves in nonlinear dispersive systems. It
turns out that the equation is not only a good model for some water waves but
also a very useful approximation model in nonlinear studies whenever one
wishes to include and balance a weak nonlinearity and weak dispersive effects
\cite{Miura}. 

\subsection{Motivation and setting of the problem} Consider the KdV equation \eqref{kdv1}. Let us introduce a source term in this equation as follows: 
\begin{equation}\label{kdv2}
u_t + 6uu_x + u_{xxx}+ f=0,
\end{equation}
where $f$ will be defined as
\begin{equation}
f:=Gu\left(  x,t\right)  =1_{\omega}\left(  u\left(  x,t\right)  -\frac
{1}{\left\vert \omega\right\vert }\int_{\omega}u\left(  x,t\right)  dx\right). \label{I11a}%
\end{equation}
Here,  $1_{\omega}$ denotes the
characteristic function of the set $\omega$. Notice that this term can be seen as a damping mechanism, which helps the energy of the system to dissipate. In fact, let us consider $\omega$ subset of a domain $\mathcal{M}:=\mathbb{T}$ or $\mathcal{M}:=\mathbb{R}$ and the total energy of the linear equation associated to \eqref{kdv2}, in this case, is given by
 \begin{equation}\label{6aa}
E_s(t) = \frac{1}{2}\int_{\mathcal{M}}\left\vert u\right\vert^2\left(  x,t\right)  dx.
\end{equation}
Then, we can (formally) verify that
 \[
\frac{d}{dt}\int_{\mathcal{M}}\left\vert u\right\vert^2\left(  x,t\right)  dx=-\left\Vert Gu\right\Vert^2_{L^2(\mathcal{M})}\text{, for any }t\in\mathbb{R}.
\]
The inequality above shows that the term $G$ plays the role of feedback mechanism and, consequently, we can investigate whether the solutions of \eqref{kdv2} tend to zero as $t\rightarrow \infty$ and under what rate they decay.

Inspired by this, in our work we will study the full KdV equation from a control point of
view posed in a bounded domain $(0,L)\subset\mathbb{R}$ with a weak forcing term $Gh$ added as a control input, namely:
\begin{equation}
\left\{
\begin{array}
[c]{lll}%
u_{t}+u_{x}+uu_{x}+u_{xxx}+Gh=0 &  &\text{ in }\left(  0,L\right)
\times\left(  0,T\right)  \text{,}\\
u\left(  0,t\right)  =u\left(  L,t\right)  =u_{x}\left(  L,t\right)
=0, &  &\text{ in } \left(  0,T\right)  \text{,}\\
u\left(  x,0\right)  =u_{0}\left(  x\right),  &  &  \text{ in }\left(
0,L\right) .
\end{array}
\right.  \label{I10}%
\end{equation}
Here, $G$ is the operator defined by
\begin{equation}
Gh\left(  x,t\right)  =1_{\omega}\left(  h\left(  x,t\right)  -\frac
{1}{\left\vert \omega\right\vert }\int_{\omega}h\left(  x,t\right)  dx\right)
, \label{I11}%
\end{equation}
where $h$ is considered as a new control input with $\omega\subset(0,L)$ and $1_{\omega}$ denotes the
characteristic function of the set $\omega$. 

Thus, we are interested in proving the stability for solutions of \eqref{I10}, which can be expressed in the following natural issue.

%

\vspace{0.2cm}

\noindent\textbf{Stabilization problem}: \textit{Can one find a feedback control law
$h$ so that the resulting closed-loop system \eqref{I10} is asymptotically stable when $t\rightarrow\infty$?}

\subsection{Previous results} The study of the controllability and stabilization to the KdV equation started with the works
of Russell and Zhang  \cite{Russell1}  for a system with periodic
boundary conditions and an internal control. Since then, both the controllability and the stabilization have
been intensively studied. In particular, the exact boundary controllability of KdV on
a finite domain was investigated in e.g.  \cite{cerpa,cerpa1,coron,GG,GG1,Rosier,Rosier2,Zhang2}.

Most of these works deal with the following system
\begin{equation}
\left\{
\begin{array}
[c]{lll}%
u_{t}+u_{x}+u_{xxx}+uu_{x}=0 &  & \text{in }(  0,T)  \times(
0,L)  \text{,}\\
u(  t,0)  =h_{1}(t),\,u(  t,L)  =h_{2}(t),\,u_{x}(
t,L)  =h_{3}(t) &  & \text{in }(  0,T), 
\end{array}
\right.  \label{I3}
\end{equation}
in which the boundary data $h_1,h_2,h_3$ can be chosen as control inputs. 

The boundary control problem of the KdV equation was first studied by Rosier \cite{Rosier} who considered system \eqref{I3} with only one boundary control input $h_{3}$ (i.e., $h_{1}=h_{2}=0$) in action. He showed that the system \eqref{I3} is locally exactly controllable in the space $L^{2}(0,L)$. Precisely, the result can be read as follows:

\vspace{0.2cm}

\noindent\textbf{Theorem $\mathcal{A}$} \cite{Rosier}: \textit{Let $T>0$ be given and assume
\begin{equation}
L\notin\mathcal{N}:=\left\{  2\pi\sqrt{\frac{j^{2}+l^{2}+jl}{3}}%
:j,l\in\mathbb{N}^{\ast}\right\}  \text{.} \label{I4}%
\end{equation}
There exists a $\delta>0$ such that if $\phi$, $\psi\in L^{2}\left(  0,L\right)  $ satisfies
\[
\left\Vert \phi\right\Vert _{L^{2}\left(  0,L\right)  }+\left\Vert
\psi\right\Vert _{L^{2}\left(  0,L\right)  }\leq\delta,
\]
then one can find a control input $h_{3}\in L^{2}\left(  0,T\right)  $ such
that the system \eqref{I3}, with $h_{1}=h_{2}=0$, admits a solution
\[
u\in C\left(  \left[  0,T\right]  ;L^{2}\left(  0,L\right)  \right)  \cap
L^{2}\left(  0,T;H^{1}\left(  0,L\right)  \right)
\]
satisfying%
\[
u\left(  x,0\right)  =\phi\left(  x\right)  \text{, }u\left(  x,T\right)
=\psi\left(  x\right)  \text{.}%
\]
}

Theorem $\mathcal{A}$ was first proved for the associated linear system using the
Hilbert Uniqueness Method due J.-L. Lions \cite{Lions} without the smallness assumption on the initial state $\phi$ and the terminal state $\psi$. The linear result was then
extended to the nonlinear system to obtain Theorem $\mathcal{A}$ by using the
contraction mapping principle.

Still regarding the KdV in a bounded domain, Chapouly \cite{Chapouly} studied  the exact controllability to the trajectories and the global exact controllability of a nonlinear KdV equation in a bounded interval. Precisely, first, she introduced two more controls as follows
\begin{equation}
\left\{
\begin{array}
[c]{lll}%
u_{t}+u_{x}+uu_{x}+u_{xxx}=g\left(  t\right), & & x\in\left(
0,L\right)  ,\text{ }t>0,\\
u\left(  0,t\right)  =h_{1}\left(  t\right)  \text{, }u\left(  L,t\right)
=h_2(t)\text{, } u_{x}\left(  L,t\right)  =0,&& t>0,
\end{array}
\right.  \label{C1a}%
\end{equation}
where $g=g(t)$ is independent of the spatial variable $x$ and is considered as
a new control input.
Then, Chapouly proved that, thanks to these three controls, the
global exact controllability to the trajectories, for any positive time T, holds. Finally, she introduced a fourth control on the first derivative at the right endpoint, namely, 
\begin{equation*}
\left\{
\begin{array}
[c]{lll}%
u_{t}+u_{x}+uu_{x}+u_{xxx}=g\left(  t\right), & & x\in\left(
0,L\right)  ,\text{ }t>0,\\
u\left(  0,t\right)  =h_{1}\left(  t\right)  \text{, }u\left(  L,t\right)
=h_2(t)\text{, } u_{x}\left(  L,t\right)  =h_3(t),&& t>0,
\end{array}
\right. 
\end{equation*}
where $g=g(t)$ has the same structure as in \eqref{C1a}. With this equation in hand, she showed the global
exact controllability, for any positive time T.

Considering now a periodic domain $\mathbb{T}$, Laurent \textit{et al.} in \cite{Laurent} worked with the following equation:
\begin{equation}
u_{t}+uu_{x}+u_{xxx}=0\text{, } \quad x\in\mathbb{T}\text{, }t\in\mathbb{R}\text{.}
\label{I5}%
\end{equation}
Equation \eqref{I5} is known to possess an infinite set of conserved integral
quantities, of which the first three are%
\[
I_{1}\left(  t\right)  =%
{\displaystyle\int_{\mathbb{T}}}
u\left(  x,t\right)  dx\text{,}%
\quad
I_{2}\left(  t\right)  =%
{\displaystyle\int_{\mathbb{T}}}
u^{2}\left(  x,t\right)  dx
\]
and
\[
I_{3}\left(  t\right)  =%
{\displaystyle\int_{\mathbb{T}}}
\left(  u_{x}^{2}\left(  x,t\right)  -\frac{1}{3}u^{3}\left(  x,t\right)
\right)  dx\text{.}%
\]

From the historical origins \cite{Boussinesq,Korteweg,Miura} of the KdV equation, involving the behavior of water waves in a shallow channel, it is natural to think of $I_{1}$ and $I_{2}$ as expressing
conservation of volume (or mass) and energy, respectively. The Cauchy problem
for equation (\ref{I5}) has been intensively studied for many years (see
\cite{Bourgain,Kato,Kenig,Saut} and the references therein). 

With respect to control theory, Laurent \textit{et al.} \cite{Laurent} studied the equation \eqref{I5} from a control point of
view with a forcing term $f=f(x,t)$ added to the equation as a control input:%
\begin{equation}
u_{t}+uu_{x}+u_{xxx}=f\text{, \ }x\in\mathbb{T}\text{, }\quad t\in\mathbb{R}\text{,}
\label{I6}%
\end{equation}
where $f$ is assumed to be supported in a given open set $\omega
\subset\mathbb{T}$. However, in the periodic domain, control problems were first studied by Russell and Zhang in \cite{Russell,Russell1}. In their works, in order to keep the mass $I_{1}(t)$
conserved, the control input $f(x,t)$ is chosen to be of the form
\begin{equation}
f\left(  x,t\right)  =\left[  Gh\right]  \left(  x,t\right)  :=g\left(
x\right)  \left(  h\left(  x,t\right)  -\int_{\mathbb{T}}g\left(  y\right)
h\left(  y,t\right)  dy\right), \label{I8}%
\end{equation}
where $h$ is considered as a new control input, and $g(x)$ is a given
non-negative smooth function such that $\{g>0\}=\omega$ and%
\[
2\pi\left[  g\right]  =\int_{\mathbb{T}}g\left(  x\right)  dx=1\text{.}%
\]

For the chosen $g$, it is easy to see that%
\[
\frac{d}{dt}\int_{\mathbb{T}}u\left(  x,t\right)  dx=\int_{\mathbb{T}}f\left(
x,t\right)  dx=0\text{, for any }t\in\mathbb{R}%
\]
for any solution $u=u(x,t)$ of the system%
\begin{equation}
u_{t}+uu_{x}+u_{xxx}=Gh. \label{I9}%
\end{equation}
Thus, the mass of the system is indeed conserved. Therefore, the following results are due to Russell and Zhang.

\vspace{0.2cm}

\noindent\textbf{Theorem $\mathcal{B}$} \cite{Russell1}: \textit{Let $s\geq0$ and $T>0$ be given. There exists a $\delta>0$ such
that for any $u_{0},u_{1}\in H^{s}(\mathbb{T})$ with $[u_{0}]=[u_{1}]$
satisfying
\[
\left\Vert u_{0}\right\Vert _{H^{s}}\leq\delta\text{, \ }\left\Vert
u_{1}\right\Vert _{H^{s}}\leq\delta,
\]
one can find a control input $h\in L^{2}(0,T;H^{s}(\mathbb{T}))$ such that the
system (\ref{I9}) admits a solution $u\in C([0,T];H^{s}(\mathbb{T}))$
satisfying $u(x,0)=u_{0}(x),u(x,T)=u_{1}(x)$.}

\vspace{0.2cm}

Note that one can always find an appropriate control input $h$ to guide system
(\ref{I6}) from a given initial state $u_{0}$ to a terminal state $u_{1}$ so
long as their amplitudes are small and $[u_{0}]=[u_{1}]$. With this result the two following questions arise naturally, which have already been cited in this work.

\vspace{0.2cm}

\noindent\textbf{Question 1}: \textit{Can one still guide the system by choosing
appropriate control input $h$ from a given initial state $u_0$ to a given terminal state $u
_1$ when $u_0$ or $u_1$ have large amplitude?}

\vspace{0.2cm}

\noindent\textbf{Question 2}: \textit{Do the large amplitude solutions of the
closed-loop system (\ref{I6}) decay exponentially as $t\rightarrow\infty$}?

Laurent \textit{et al.}  gave the positive answers to these questions:

\vspace{0.2cm}

\noindent\textbf{Theorem $\mathcal{C}$} \cite{Laurent}: \textit{Let $s\geq0$, $R>0$ and $\mu\in\mathbb{R}$ be given. There exists
a $T>0$ such that for any $u_{0},u_{1}\in H^{s}(\mathbb{T})$ with
$[u_{0}]=[u_{1}]=\mu$ are such that
\[
\left\Vert u_{0}\right\Vert _{H^{s}}\leq R\text{, \ }\left\Vert u_{1}%
\right\Vert _{H^{s}}\leq R,
\]
then one can find a control input $h\in L^{2}(0,T;H^{s}(\mathbb{T}))$ such
that the system (\ref{I6}) admits a solution $u\in C([0,T];H^{s}(\mathbb{T}))$
satisfying
\[
u(x,0)=u_{0}(x)\quad \text{and}\quad u(x,T)=u_{1}(x).
\]}

\vspace{0.2cm}

\noindent\textbf{Theorem $\mathcal{D}$} \cite{Laurent}:   \textit{Let $s\geq0$, $R>0$ and $\mu\in\mathbb{R}$ be given. There exists
a $k>0$ such that for any $u_{0}\in H^{s}(\mathbb{T})$ with $[u_{0}]=\mu$ the
corresponding solution \ of the system (\ref{I6}) satisfies%
\[
\left\Vert u\left(  \cdot,t\right)  -\left[  u_{0}\right]  \right\Vert
_{H^{s}}\leq\alpha_{s,\mu}\left(  \left\Vert u_{0}-\left[  u_{0}\right]
\right\Vert _{H^{0}}\right)  e^{-kt}\left\Vert u_{0}-\left[  u_{0}\right]
\right\Vert _{H^{s}}\text{ for all }t>0\text{,}%
\]
where $\alpha_{s,\mu}:\mathbb{R}^{+}\longrightarrow\mathbb{R}^{+}$ is a
nondecreasing continuous function depending on $s$ and $\mu$.}

\vspace{0.2cm}

These results are established with the aid of certain properties of propagation of
compactness and regularity in Bourgain spaces for the solutions of the associated
linear system. Finally, with Slemrod’s feedback law, the resulting closed-loop system is
shown to be locally exponentially stable with an arbitrarily large decay rate. 

Still with respect to problems of stabilization, Pazoto \cite{Pazoto} proved the exponential decay for the energy of solutions of the Korteweg-de Vries equation in a bounded interval with a localized damping term, precisely, with a term $a=a(x)$ satisfying 
\begin{equation}\label{pazoto}
\left\{\begin{array}{l}
a \in L^{\infty}(0, L) \text { and } a(x) \geq a_{0}>0 \text { a.e. in } \omega, \\
\text { where } \omega \text { is a nonempty open subset of }(0, L).
\end{array}\right.
\end{equation}
With this mechanism the author showed that 
$$
\frac{\mathrm{d} E}{\mathrm{d} t}=-\int_{0}^{L} a(x)|u(x, t)|^{2} \mathrm{d} x-\frac{1}{2}\left|u_{x}(0, t)\right|^{2}
$$
with
$$
E(t)=\frac{1}{2} \int_{0}^{L}|u(x, t)|^{2} \mathrm{d} x.
$$
This indicates that the term $a(x) u$ in the equation plays the role of a feedback damping mechanism. Finally, following the method in Menzala \textit{et al.} \cite{menzala}  which combines energy estimates, multipliers and compactness arguments, the problem is reduced to prove the unique continuation of weak solutions. The result proved by the author can be read as follows.

\vspace{0.2cm}

\noindent\textbf{Theorem $\mathcal{E}$} \cite{Pazoto}: \textit{For any $L>0,$ any damping potential a satisfying \eqref{pazoto} and $R>0,$ there exist $c=c(R)>0$ and $\mu=\mu(R)>0$ such that
$$
E(t) \leq c\left\|u_{0}\right\|_{L^{2}(0, L)}^{2} \mathrm{e}^{-\mu t},
$$
holds for all $t \geq 0$ and any solution of
\begin{equation}\label{pazoto2}
\left\{
\begin{array}
[c]{lll}%
u_{t}+u_{x}+uu_{x}+u_{xxx}+a(x)u=0 &  & \text{in
}(  0,T)  \times(  0,L)  \text{,}\\
u(  t,0)  =u(  t,L)  =u_{x}(  t,L)  =0 &  &
\text{in }(  0,T)  \text{,}\\
u(  0,x)  =u_{0}(  x)  &  & \text{in }(
0,L)  \text{,}%
\end{array}
\right.
\end{equation}
with $u_{0} \in L^{2}(0, L)$ such that $\left\|u_{0}\right\|_{L^{2}(0, L)} \leq R$.}

\vspace{0.2cm}

Massarolo \textit{et al.}  showed in \cite{massarolo} that a very weak amount of
additional damping stabilizes the KdV equation. In particular, a damping mechanism dissipating the $L^2-$norm as $a(\dot)$ does is not needed. Dissipating the $H^{-1}-$norm proves to be. For instance, one can take the damping term $Bu$ instead of $a(x)u$, where $Bu$ is defined by
$$
B=1_{\omega}\left(-\frac{\mathrm{d}^{2}}{\mathrm{d} x^{2}}\right)^{-1}=1_{\omega}(-\Delta)^{-1},
$$
where $1_{\omega}$ denotes the characteristic function of the set $\omega,\left(-d^{2} / d x^{2}\right)^{-1}$ is the inverse of the Laplace operator with Dirichlet boundary conditions (on the boundary of $\omega\subset(0,L)$). Under the above considerations, they observed that (formally) the operator $B$ satisfies
$$
\begin{aligned}
\int_{0}^{L} u B u \mathrm{d} x &=\int_{0}^{L} u\left[-1_{\omega} \Delta^{-1} u\right] \mathrm{d} x=-\int_{\omega}\left(\Delta^{-1} u\right) \Delta\left(\Delta^{-1} u\right) \mathrm{d} x \\
&=-\left.\Delta^{-1} u\left[\Delta^{-1} u\right]_{x}\right|_{\partial \omega}+\int_{\omega}\left|\left[\Delta^{-1} u\right]_{x}\right|^{2} \mathrm{d} x \\
&=\left\|\left[\Delta^{-1} u\right]_{x}\right\|_{L^{2}(\omega)}^{2}=\left\|\Delta^{-1} u\right\|_{H_{0}^{1}(\omega)}^{2}=\|u\|_{H^{-1}(\omega)}^{2}.
\end{aligned}
$$
Consequently, the total energy $E(t)$ associated with \eqref{pazoto2} with $Bu$ instead of $a(x)u$, satisfies
$$
\frac{\mathrm{d}}{\mathrm{d} t} \int_{0}^{L}|u(x, t)|^{2} \mathrm{d} x=-u_{x}^{2}(0, t)-\|u\|_{H^{-1}(\omega)}^{2},
$$
where
$$
E(t)=\int_{0}^{L}|u(x, t)|^{2} \mathrm{d} x.
$$
This indicates that the term $Bu$ plays the role of a feedback damping mechanism. Consequently, they  investigated whether $E(t)$ tends to zero as $t \rightarrow \infty$ and the uniform rate at which it may decay, showing the similar result as in Theorem $\mathcal{E}$.

To finish that small sample of the previous works, let us present another result of controllability for the KdV equation posed on a bounded domain. Recently, the author in collaboration with Pazoto and Rosier, showed in \cite{CaRoPa}  results for the following system,
\begin{equation}
\left\{
\begin{array}
[c]{lll}%
u_{t}+u_{x}+uu_{x}+u_{xxx}=1_{\omega}f(  t,x)  &  & \text{in
}(  0,T)  \times(  0,L)  \text{,}\\
u(  t,0)  =u(  t,L)  =u_{x}(  t,L)  =0 &  &
\text{in }(  0,T)  \text{,}\\
u(  0,x)  =u_{0}(  x)  &  & \text{in }(
0,L)  \text{,}%
\end{array}
\right. \label{ncn2}%
\end{equation}
considering $f$ as a control input and $1_{\omega}$ is a characteristic function supported on $\omega\subset (0,L)$.  Precisely, when the control acts in a neighborhood of $x = L$,  they obtained the exact controllability in the weighted Sobolev space $L^{2}_{\frac{1}{L-x}  dx}$ defined as
\[
L^{2}_{\frac{1}{L-x}  dx}:=\{  u\in L^1_{loc}(0,L);
\int_{0}^{L}\frac{\left\vert u(x)
\right\vert ^{2}}{ L-x  } dx<\infty\} .
\]
More precisely, they proved the following result:
\vspace{0.2cm}

\noindent\textbf{Theorem $\mathcal{F}$} \cite{CaRoPa}:   \textit{Let $T>0$, $\omega=(l_1,l_2)=(L-\nu ,L)$ where $0<\nu <L$.
Then, there exists $\delta>0$ such that for any $u_{0}$, $u_{1}\in L^2_{  \frac{1}{L-x}  dx}$
with
\[
\left\Vert u_{0}\right\Vert _{L^2_{ \frac{1}{L-x}  dx} } \leq\delta \ \text{ and } \  
\left\Vert u_{1}\right\Vert _{L^2_{  \frac{1}{L-x}  dx}}\leq\delta ,
\]
one can find a control input $f\in L^{2}(  0,T;H^{-1}(0,L))  $ with $\text{supp} (f)\subset (0,T) \times \omega$ 
such that the solution $u\in C^0([0,L],L^2 (0,L) )
\cap L^2(0,T,H^1(0,L))$ of (\ref{ncn2}) satisfies $u(T,.)  =u_{1}\,\text{ in } (0,L)$ and $u\in C^0([0,T],L^2_{  \frac{1}{L-x}  dx} )$. Furthermore, 
$f\in L^2_{ (T-t) dt}(0,T,L^2(0,L))$.}

\vspace{0.2cm}

We caution that this is only a small sample of the extant works in this field. Now, we are able to present our result in this manuscript.

\subsection{Main result and heuristic of the paper}  The aim of this manuscript is to address the stabilization issue  for the KdV equation on a bounded domain with a \textit{weak source (or forcing) term}, as a distributed control, namely 
\begin{equation}
\left\{
\begin{array}
[c]{lll}%
u_{t}+u_{x}+uu_{x}+u_{xxx}+Gh=0, && \text{in }\left(  0,L\right)
\times\left(  0,T\right)  \text{,}\\
u\left(  0,t\right)  =u\left(  L,t\right)  =u_{x}\left(  L,t\right)
=0, && \text{in }\left(  0,T\right)  \text{,}\\
u\left(  x,0\right)  =u_{0}\left(  x\right), &&  \text{in }\left(
0,L\right)  ,
\end{array}
\right.  \label{I10aa}%
\end{equation}
where $G$ is the operator defined by \eqref{I11}. 

Notice that with a good choose of $Gh$, that is, 
\begin{equation}\label{h2}
Gh:=Gu\left(  x,t\right)  =1_{\omega}\left(  u\left(  x,t\right)  -\frac
{1}{\left\vert \omega\right\vert }\int_{\omega}u\left(  x,t\right)  dx\right),
\end{equation}
the energy associate 
\[
I_{2}\left(  t\right)  =%
{\displaystyle\int_0^L}
u^{2}\left(  x,t\right)  dx
\]
verify that
 \[
\frac{d}{dt}\int_0^L u^2\left(  x,t\right)  dx\leq-\left\Vert Gu\right\Vert^2_{L^2(0,L)}\text{, for any }t>0,
\]
at least for the linear system $$u_{t}+u_{x}+u_{xxx}+Gh=0, \quad \text{ in } (0,L)\times\{t>0\}.$$ Consequently, we can investigate whether the solutions of  this equation tend to zero as $t\rightarrow \infty$ and under what rate they decay. To be precise, the main result of the work, give us an answer to the stabilization problem for the system \eqref{I10}-\eqref{I11}, proposed on the beginning of this paper, and will be state in the following form.
\begin{theorem}
\label{main1} Let $T>0$. Then, for every $R_{0}>0$ there exist
 constants $C>0$ and $k>0$, such that, for any $u_{0}\in L^{2}\left(  0,L\right)  $ with%
\[
\left\Vert u_{0}\right\Vert _{L^{2}\left(  0,L\right)  }\leq R_{0}\text{,}%
\]
the corresponding solution $u$ of (\ref{I10}) satisfies%
\[
\left\Vert u\left(  \cdot,t\right)  \right\Vert _{L^{2}\left(  0,L\right)
}\leq Ce^{-kt}\left\Vert u_{0}\right\Vert _{L^{2}\left(  0,L\right)}, \quad \forall t>0.
\]
\end{theorem}

Note that our goal in this work is to give an answer for the stabilization problem that was mentioned at the beginning of this introduction.  Is important to point out that a similar feedback law was used in \cite{Russell1} and, more recently, in \cite{Laurent} for the Korteweg-de Vries equation, to prove a globally uniform exponential result in a periodic domain. In \cite{Laurent,Russell1} the damping with a null mean was introduced to conserve the integral of the solution, which for KdV represents the mass (or volume) of the fluid.

In the context presented in this manuscript, our result improves earlier works on the subject, for example, \cite{Pazoto}. Roughly speaking, differently from what was proposed by \cite{Laurent,Russell1}, in this work,  the weak damping \eqref{I11} is to have a lower cost than the one presented in \cite{Pazoto} in the sense of that we can remove a medium term in the mechanisms proposed in these works and still have positive result of stabilization of the KdV equation.

Observe that the control used in \cite{Pazoto}, is formally the first part of the following forcing term:
$$
Gh\left(  x,t\right)  =1_{\omega}\left(  h\left(  x,t\right)  -\frac
{1}{\left\vert \omega\right\vert }\int_{\omega}h\left(  x,t\right)  dx\right),
$$
where $\omega\subset(0,L)$. In fact, to see this, in \cite{Pazoto}, define $a(x):=-1_{\omega}$ in the above equality and just forget the remaining term. Thus, due to these considerations,  we do not need a strong mechanism acting as control input. Surely, of what was shown in this article, to achieve the stability result for the KdV equation, is that the forcing operator $Gh$ can be taken as a function supported in $\omega$ removing the medium term associated to the first term of the control mechanism.  

Here, it is important point out that, the week damping mechanism is related with respect to the cost of the stabilization, as mentioned previously, which is different in the context of \cite{massarolo}, where the authors proves that the energy of the system dissipates in the $H^{-1}-$norm instead of $L^2-$norm.

Concerning to the stabilization problem, the main ingredient to prove Theorem \ref{main1} is the \textit{Carleman estimate} for the linear problem proved by Capistrano-Filho \textit{et al.} in \cite{CaRoPa}. This estimate together with the energy estimate and compactness arguments reduces the problem to prove the \textit{Unique Continuation Property (UCP)} for the solutions of the nonlinear problem, precisely, the  following result is showed.

\vspace{0.2cm}

\noindent\textbf{UCP:} \textit{Let $L>0$ and $T>0$ be two real numbers, and let $\omega\subset\left(
0,L\right)  $ be a nonempty open set. If $v\in L^{\infty}\left(
0,T;H^{1}\left(  0,L\right)  \right)  $ solves%
\[
\left\{
\begin{array}
[c]{lll}%
v_{t}+v_{x}+v_{xxx}+vv_{x}=0,&  & \text{in }\left(  0,L\right)  \times\left(
0,T\right)  ,\\
v\left(  0,t\right)  =v\left( L,t\right)  =0,&  & \text{in }\left(
0,T\right)  ,\\
v=c, &  & \text{in } \omega\times\left(  0,T\right)  \text{,}%
\end{array}
\right.
\]
for some $c\in\mathbb{R}$. Thus, $v\equiv c$ in $\left(  0,L\right)  \times\left(  0,T\right)  $, where $c\in\mathbb{R}$.}

It is important to point out here that the previous UCP was first proved by Rosier and Zhang in \cite{Rosier3}. In this way, to sake of completeness, we revisited this result now using the Carleman estimate proved by the author in \cite{CaRoPa}.

\subsection{Structure of the work} To end our introduction, we present the outline of the manuscript: In Section \ref{Sec1}, we present some estimates for the KdV equation which will be used in the course of the work.  Section \ref{Sec3} is devoted to present the proof of Theorem \ref{main1}, that is, give the answer to the stabilization problem. Comments of our result as well as some extensions for other models are presented in Section \ref{Sec4}. Finally, on the Appendix \ref{Apendice}, we will give a sketch how to prove the unique continuation property (UCP) presented above.

\section{Well-posedness for KdV equation}\label{Sec1}
In this section, we will review a series of estimates for the KdV equation, namely, 
\begin{equation}
\left\{
\begin{array}
[c]{lll}%
u_{t}+u_{x}+uu_{x}+u_{xxx}=f,&& \text{in } \left(  0,L\right)
\times\left(  0,T\right)\text{,}\\
u\left(  0,t\right)  =u\left(  L,t\right)  =u_{x}\left(  L,t\right)
=0, &&\text{in }\left(  0,T\right)\text{,}\\
u\left(  x,0\right)  =u_{0}\left(  x\right), && \text{in }\left(
0,L\right),
\end{array}
\right.  \label{I10a}%
\end{equation}
which will be borrowed of \cite{Rosier}. Here $f = f(t,x)$ is a function which stands for the control of the system.

\subsection{The linearized KdV equation}

The well-posedness of the problem \eqref{I10a}, with $f\equiv0$, was proved by Rosier \cite{Rosier}. He notice that operator $A=-\dfrac{\partial^{3}%
}{\partial x^{3}}-\dfrac{\partial}{\partial x}$ with domain
\[
D\left(  A\right)  =\left\{  w\in H^{3}\left(  0,L\right)  ;w\left(  0\right)
=w\left(  L\right)  =w_{x}\left(  L\right)  =0\right\}  \subseteq L^{2}\left(
0,L\right)
\]
is the infinitesimal generator of a strongly continuous semigroup of contractions in  $L^{2}\left(  0,L\right)  $.

\begin{theorem}
\label{linkdv}Let $u_{0}\in L^{2}\left(  0,L\right)  $ and $f\equiv0$. There
exists a unique weak solution $u=S\left(  \cdot\right)  u_{0}$ of \eqref{I10a}
such that%
\begin{equation}
u\in C([0,T];L^{2}(0,L))\cap H^{1}(0,T;H^{-2}\left(  0,L\right)  )\text{.}
\label{newh}%
\end{equation}
Moreover, if $u_{0}\in D\left(  A\right)  $, then \eqref{I10a} has a unique
(classical ) solution $u$ such that%
\begin{equation}
u\in C([0,T];D(A))\cap C^{1}(0,T;L^{2}(0,L))\text{.} \label{new_h}%
\end{equation}

\end{theorem}

An additional regularity result for the weak solutions of the linear system
associated to system \eqref{I10a}  was also established in \cite{Rosier}. The result can be read as follows.

\begin{theorem}
\label{linkdv1}Let $u_{0}\in L^{2}\left(  0,L\right)  $, $Gw\equiv0$ and
$u=S\left(  \cdot\right)  u_{0}$ the weak solution of \eqref{I10a}. Then, $u\in$
$L^{2}(0,T;H^{1}(0,L))$ and there exists a positive constant $c_{0}$ such
that
\begin{equation}
\left\Vert u\right\Vert _{L^{2}(0,T;H^{1}(0,L))}\leq c_{0}\left\Vert
u_{0}\right\Vert _{L^{2}\left(  0,L\right)  }\text{.} \label{new_h1}%
\end{equation}
Moreover, there exist two positive constants $c_{1}$ and $c_{2}$ such that%
\begin{equation}
\left\Vert u_{x}\left(  \cdot,0\right)  \right\Vert _{L^{2}\left(  0,T\right)
}^{2}\leq c_{1}\left\Vert u_{0}\right\Vert _{L^{2}\left(  0,L\right)  }
\label{new_h2}%
\end{equation}
and%
\begin{equation}
\left\Vert u_{0}\right\Vert _{L^{2}\left(  0,L\right)  }\leq\frac{1}%
{T}\left\Vert u\right\Vert _{L^{2}\left(  0,T;L^{2}\left(  0,L\right)
\right)  }^{2}+c_{2}\left\Vert u_{x}\left(  \cdot,0\right)  \right\Vert
_{L^{2}\left(  0,T\right)  }^{2}\text{.} \label{new_h3}%
\end{equation}

\end{theorem}

\subsection{The nonlinear KdV equation}
In this section we prove the well-posedness of the following system%
\begin{equation}
\left\{
\begin{array}
[c]{lll}%
u_{t}+u_{x}+uu_{x}+u_{xxx}=Gw, &&\text{in }\left(  0,L\right)
\times\left(  0,T\right)\text{,}\\
u\left(  0,t\right)  =u\left(  L,t\right)  =u_{x}\left(  L,t\right)
=0,&&\text{in }(0,T),\\
u\left(  x,0\right)  =u^{0}\left(  x\right), &&  \text{in }\left(
0,L\right)  .
\end{array}
\right.  \label{n5a}%
\end{equation}

To solve the
problem we write the solution of \eqref{n5a} as follows $$u=S\left(  t\right)  u_{0}+u_{1}+u_{2},$$ where $\left(
S\left(  t\right)  \right)  _{t\geq0}$ denotes the semigroup associated with
the operator $Au=-u^{\prime\prime\prime}-u^{\prime}$ with domain 
$\mathcal{D}\left(  A\right)$ dense in  $ L^{2}\left(  0,L\right)  $ defined by
\[
\mathcal{D}\left(  A\right)  =\left\{  v\in H^{3}\left(  0,L\right)  ;v\left(
0\right)  =v\left(  L\right)  =v^{\prime}\left(  L\right)  =0\right\}
\text{,}%
\]
and $u_{1}$ and $u_{2}$ are (respectively) solutions of two non-homogeneous
problems%
\begin{equation}
\left\{
\begin{array}
[c]{lll}%
u_{1t}+u_{1x}+u_{1xxx}=Gw,&&\text{in }\omega\times\left(  0,T\right)
\text{,}\\
u_{1}\left(  0,t\right)  =u_{1}\left(  L,t\right)  =u_{1x}\left(  L,t\right)
=0, &&\text{in }(0,T),\\
u_{1}\left(  x,0\right)  =0, &&\text{in }\left(  0,L\right),
\end{array}
\right.  \label{n6a}%
\end{equation}
and%
\begin{equation}
\left\{
\begin{array}
[c]{lll}%
u_{2t}+u_{2x}+u_{2xxx}=f, &&\text{in }\left(  0,L\right)  \times\left(
0,T\right)\text{,}\\
u_{2}\left(  0,t\right)  =u_{2}\left(  L,t\right)  =u_{2x}\left(  L,t\right)
=0,&&\text{in } (0,T)\text{,}\\
u_{2}\left(  x,0\right)  =0, &&\text{in }\left(  0,L\right)  \text{,}%
\end{array}
\right.  \label{n7a}%
\end{equation}
where $f=-u_{2}u_{2x}$ and $w$ is solution of the following adjoint system
\begin{equation}
\left\{
\begin{array}
[c]{lll}%
-w_{t}-w_{x}-w_{xxx}=0, &&\text{in }\left(  0,L\right)  \times\left(
0,T\right)\text{,}\\
w\left(  0,t\right)  =w\left(  L,t\right)  =w_{x}\left(  0,t\right)
=0, &&\text{in }(0,T)\text{,}\\
w\left(  x,T\right)  =0\left(  x\right),&&\text{in } \left(
0,L\right)  .
\end{array}
\right.  \label{n7aa}%
\end{equation}

Let us define the following map
\[
\Psi:w\in L^{2}\left(  0,T;L^{2}\left(  0,L\right)  \right)  \longmapsto
u_{1}\in C\left(  \left[  0,T\right]  ;L^{2}\left(  0,L\right)  \right)  \cap
L^{2}\left(  0,T;H^{1}\left(  0,L\right)  \right)  :=B\text{,}%
\]
endowed with norm%
\[
\left\Vert u_{1}\right\Vert _{B}:=\sup\limits_{t\in\left[  0,T\right]
}\left\Vert u_{1}\left(  \cdot,t\right)  \right\Vert _{L^{2}\left(
0,L\right)  }+\left(  \int_{0}^{T}\left\Vert u_{1}\left(  \cdot,t\right)
\right\Vert _{H^{1}\left(  0,L\right)  }^{2}dt\right)  ^{\frac{1}{2}}\text{,}%
\]
be the map which associates with $w$ the weak solution of (\ref{n6a}). Observe
that, by using Theorem \ref{linkdv1}  the map $u_{0}\in L^{2}\left(  0,L\right)
\mapsto S\left(  \cdot\right)  u^{0}\in B$ is continuous. Furthermore, the following proposition holds true. 

\begin{proposition}
\label{prop4} The function $\Psi$ is a (linear) continuous map.
\end{proposition}

\begin{proof} Indeed, let us divide the proof in two parts.

\vspace{0.1cm}

\noindent\textbf{First part.} 

\vspace{0.1cm}

Notice that in \eqref{n6a} $w$ is the solution of \eqref{n7aa}, thus,  $g\left(  x,t\right)  =Gw\left(  x,t\right)  \in
C^{1}\left(  \left[  0,T\right]  ;L^{2}\left(  0,L\right)  \right)  $ and from
classical results concerning such non-homogeneous problems (see \cite{pazy}) we
obtain a unique solution%
\begin{equation}
u_{1}\in C\left(  \left[  0,T\right]  ;\mathcal{D}\left(  A\right)  \right)
\cap C^{1}\left(  \left[  0,T\right]  ;L^{2}\left(  0,L\right)  \right)
\label{n9}%
\end{equation}
of \eqref{n6a}. Additionally, the following estimate can be proved:
\begin{equation}
\int^{T}_{0}\left\|Gu\right\|_{L^{2}(0,L)}dt\leq CT \left\|u\right\|_{Y_{0,T}},\label{tech2.1}
\end{equation}
where, 
$$Y_{0,T}=C([0,T];L^2(0,T))\cap L^{2}([0,T];H^1(0,L)).$$

In fact,  by a direct computation, we have
\begin{align*}
\int^{T}_{0}||Gu||^2_{L^2(0,L)}dt&=\int^T_0\big(\int_{\omega}u^2dx-|\omega|^{-1}\big(\int_{\omega}udx\big)^2\big)^{1/2}dt\\
&\leq\int^T_0\big(\int^L_0u^2dx\big)^{1/2}dt\leq T||u||_{Y_{0,T}}.
\end{align*}
Thus, \eqref{tech2.1} follows.

\vspace{0.1cm}

\noindent\textbf{Second part.} 

\vspace{0.1cm}

Now, we will prove some estimates by multipliers method. Consider
$u_{0}\left(  x\right)  \in\mathcal{D}\left(  A\right)  $. Let $w\in
L^{2}\left(  0,T;L^{2}\left(  0,L\right)  \right)  $ and $q\in C^{\infty
}\left(  \left[  0,T\right]  \times\left[  0,L\right]  \right)  $. Multiplying
(\ref{n6a}) by $qu_{1}$, we obtain%
\begin{equation}
\int_{0}^{S}\int_{0}^{L}qu_{1}\left(  u_{1t}+u_{1x}+u_{1xxx}\right)
dxdt=\int_{0}^{S}\int_0^L qu_{1}\left(  Gw\right)  dxdt\text{,}
\label{n10}%
\end{equation}
where $S\in\left[  0,T\right]  $. Using (\ref{n6a}) (and Fubini's theorem) we
get:%
\begin{equation}%
\begin{array}
[c]{c}%
-%
{\displaystyle\int_{0}^{S}}
{\displaystyle\int_{0}^{L}}
\left(  q_{t}+q_{x}+q_{xxx}\right)  \dfrac{u_{1}^{2}}{2}dxdt+%
{\displaystyle\int_{0}^{L}}
\left(  \dfrac{qu_{1}^{2}}{2}\right)  \left(  x,S\right)  dx\\{\displaystyle
+\frac{3}{2}}
{\displaystyle\int_{0}^{S}}
{\displaystyle\int_{0}^{L}}
q_{x}u_{1x}^{2}dxdt+{\displaystyle\frac{1}{2}}
{\displaystyle\int_{0}^{S}}
\left(  qu_{1x}^{2}\right)  \left(  0,t\right)  dt=%
{\displaystyle\int_{0}^{S}}
{\displaystyle\int_0^L}
\left(  qu_{1}\right)  \left(  Gw\right)  dxdt\text{.}%
\end{array}
\label{n11}%
\end{equation}
Choosing $q=1$ it follows that%
\[%
\begin{array}
[c]{lll}%
{\displaystyle\int_{0}^{L}}
u_{1}\left(  x,S\right)  ^{2}dx+%
{\displaystyle\int_{0}^{S}}
u_{1x}\left(  0,t\right)  ^{2}dt & = &
{\displaystyle\int_{0}^{S}}
{\displaystyle\int_{0}^L}
u_{1}\left(  Gw\right)  dxdt\\
& \leq & \dfrac{1}{2}\left\Vert u\right\Vert _{L^{2}\left(  0,S;L^{2}\left(
0,L\right)  \right)  }+\dfrac{1}{2}\left\Vert Gw\right\Vert _{L^{2}\left(
0,S;L^{2}\left( 0,L\right)  \right)  }^{2}\text{.}%
\end{array}
\]
Then, we get%
\begin{equation}
\left\Vert u_{1}\right\Vert _{C\left(  \left[  0,T\right]  ;L^{2}\left(
0,L\right)  \right)  }\leq C\left\Vert Gw\right\Vert _{L^{2}\left(
0,T;L^{2}\left(  0,L\right)  \right), } \label{n12}%
\end{equation}
which yields
\begin{equation}
\left\Vert u_{1}\right\Vert _{L^{2}\left(  \left(  0,T\right)  \times\left(
0,L\right)  \right)  }\leq C\left\Vert Gw\right\Vert _{L^{2}\left(
0,T;L^{2}\left(  0,L\right)  \right)  } \label{n121}%
\end{equation}
and%
\begin{equation}
\left\Vert u_{1x}\left(  0,\cdot\right)  \right\Vert _{L^{2}\left(
0,T\right)  }\leq C\left\Vert Gw\right\Vert _{L^{2}\left(  0,T;L^{2}\left(
0,L\right)  \right)  }\text{.} \label{n13}%
\end{equation}
Now take $q\left(  x,t\right)  =x$ and $S=T$, (\ref{n11}) gives,%
\begin{equation}
-\int_{0}^{T}\int_{0}^{L}\frac{u_{1}^{2}}{2}dxdt+\int_{0}^{L}\frac{x}{2}%
u_{1}^{2}\left(  x,T\right)  dx+\frac{3}{2}\int_{0}^{T}\int_{0}^{L}u_{1x}%
^{2}dxdt=\int_{0}^{T}\int_{0}^Lxu_{1}\left(  Gw\right)  dxdt\text{.}
\label{n14}%
\end{equation}
Hence%
\[
\int_{0}^{T}\int_{0}^{L}u_{1x}^{2}dxdt\leq\frac{1}{3}\left(  \int_{0}^{T}%
\int_{0}^{L}u_{1}^{2}dxdt+L\left\{  \int_{0}^{T}\int_{0}^{L}u^{2}dxdt+\int
_{0}^{T}\int_{0}^L\left(  Gw\right)  ^{2}dxdt\right\}  \right)
\]
and then, using (\ref{n121}),
\begin{equation}
\left\Vert u_{1}\right\Vert _{L^{2}\left(  0,T;H^{1}\left(  0,L\right)
\right)  }\leq C\left(  T,L\right)  \left\Vert Gw\right\Vert _{L^{2}\left(
0,T;L^{2}\left(  0,L\right)  \right)  }\text{.} \label{n15}%
\end{equation}
Using (\ref{n12}), (\ref{n15}), \eqref{tech2.1} and the density of $\mathcal{D}\left(
A\right)  $ in $L^{2}\left(  0,L\right)  $, we deduce that $\Psi$ is a
linear continuous map, proving thus the proposition.
\end{proof}

The next result, proved in \cite[Proposition 4.1]{Rosier}, give us that nonlinear system \eqref{n7a} is well-posed.

\begin{proposition}\label{prop5}The following items can be proved.
\begin{itemize} 
\item[1.]  If $u\in L^{2}\left(  0,T;H^{1}\left(  0,L\right)  \right)  $,
$uu_{x}\in L^{1}\left(  0,T;L^{2}\left(  0,L\right)  \right)  $ and $u\mapsto
uu_{x}$ is continuous.
\item[2.] For $f\in L^{1}\left(  0,T;L^{2}\left(  0,L\right)  \right)  $ the mild
solution $u_{2}$ of \eqref{n7a} belongs to $B$. Moreover, the linear map%
\[
\Theta:f\longmapsto u_{2}%
\]
is continuous.
\end{itemize}
\end{proposition}

\begin{remark}
\label{rmk2}Recall that for $f\in L^{1}\left(  0,T;L^{2}\left(  0,L\right)
\right)  $ the mild solution $u_{2}$ of \eqref{n7a} is given by%
\begin{equation}
u_{2}\left(  \cdot,t\right)  =\int_{0}^{t}S\left(  t-s\right)  f\left(
\cdot,s\right)  ds\text{.} \label{n8}%
\end{equation}
\end{remark}

\section{Stabilization of KdV equation\label{Sec3}}

In this section we study the stabilization of the system%
\begin{equation}
\left\{
\begin{array}
[c]{lll}%
u_{t}+u_{x}+uu_{x}+u_{xxx}+Gu=0,&&\text{in }\left(  0,L\right)
\times\{t>0\}\text{,}\\
u\left(  0,t\right)  =u\left(  L,t\right)  =u_{x}\left(  L,t\right)
=0, && t>0,\\
u\left(  x,0\right)  =u_{0}\left(  x\right), &&  \text{in }\left(
0,L\right).
\end{array}
\right.  \label{m1}%
\end{equation}
Here, $Gu$ is defined by \eqref{h2}.  Precisely, the issue in this section is the following one:

\vspace{0.2cm}

\noindent\textbf{Stabilization problem:} \textit{Can one find a feedback control law
$h$ so that the resulting closed-loop system \eqref{m1} is asymptotically stable when $t\rightarrow\infty$?}

\vspace{0.2cm}

The answer to the stability problem is given by the theorem below.

\begin{theorem}
\label{stabilization} Let $T>0$. Then, there exist constants $k>0$, $R_{0}>0$ and $C>0$, such that for any $u_{0}\in
L^{2}\left(  0,L\right)  $ with%
\[
\left\Vert u_{0}\right\Vert _{L^{2}\left(  0,L\right)  }\leq R_{0}\text{,}%
\]
the corresponding solution $u$ of (\ref{m1}) satisfies%
\begin{equation}
\left\Vert u\left(  \cdot,t\right)  \right\Vert _{L^{2}\left(  0,L\right)
}\leq Ce^{-kt}\left\Vert u_{0}\right\Vert _{L^{2}\left(  0,L\right)  }\text{,
}\forall t\geq0\text{.} \label{m3}%
\end{equation}
\end{theorem}

As usual in the stabilization problem, Theorem \ref{stabilization} is a direct consequence of the following
\textit{observability inequality}.

\begin{proposition}\label{OI}
Let $T>0$ and $R_{0}>0$ be given. There exists a constant $C>1$, such that,
for any $u_{0}\in L^{2}\left(  0,L\right)  $ satisfying%
\[
\left\Vert u_{0}\right\Vert _{L^{2}\left(  0,L\right)  }\leq R_{0}\text{,}%
\]
the corresponding solution $u$ of (\ref{m1}) satisfies%
\begin{equation}
\left\Vert u_{0}\right\Vert _{L^{2}\left(  0,L\right)  }^{2}\leq C\int_{0}%
^{T}\left\Vert Gu\right\Vert _{L^{2}\left(  0,L\right)  }^{2}dt\text{.}
\label{m4}%
\end{equation}

\end{proposition}

Indeed, if (\ref{m4}) holds, then it follows from the energy estimate that
\begin{equation}
\left\Vert u\left(  \cdot,T\right)  \right\Vert _{L^{2}\left(  0,L\right)
}^{2}\leq\left\Vert u_{0}\right\Vert _{L^{2}\left(  0,L\right)  }^{2}-\int
_{0}^{T}\left\Vert Gu\right\Vert _{L^{2}\left(  0,L\right)  }^{2}dt\text{,}
\label{m5}%
\end{equation}
or, more precisely,
\[
\left\Vert u\left(  \cdot,T\right)  \right\Vert _{L^{2}\left(  0,L\right)
}^{2}\leq\left(  1-C^{-1}\right)  \left\Vert u_{0}\right\Vert _{L^{2}\left(
0,L\right)  }^{2}\text{.}%
\]
Thus,%
\[
\left\Vert u\left(  \cdot,mT\right)  \right\Vert _{L^{2}\left(  0,L\right)
}^{2}\leq\left(  1-C^{-1}\right)  ^{m}\left\Vert u_{0}\right\Vert
_{L^{2}\left(  0,L\right)  }^{2}%
\]
which gives (\ref{m3}) by the semigroup property. In (\ref{m3}), we obtain a
constant $k$ independent of $R_{0}$ by noticing that for $t>c\left(
\left\Vert u_{0}\right\Vert _{L^{2}\left(  0,L\right)  }\right)  $, the
$L^{2}-$ norm of $u\left(  \cdot,t\right)  $ is smaller than $1$, so that we
can take the $k$ corresponding to $R_{0}=1$.

\begin{proof}[Proof of Proposition \ref{OI}]We prove (\ref{m4}) by contradiction. Suppose that \eqref{m4} does not occurs. Thus, for any $n\geq1$, (\ref{m1}) admits a solution $u_{n}\in
C\left(  \left[  0,T\right]  ;L^{2}\left(  0,L\right)  \right)  \cap
L^{2}\left(  0,T;H^{1}\left(  0,L\right)  \right)  $ satisfying%
\[
\left\Vert u_{n}\left(  0\right)  \right\Vert _{L^{2}\left(  0,L\right)  }\leq
R_{0}\text{,}%
\]
and%
\begin{equation}
\int_{0}^{T}\left\Vert Gu_{n}\right\Vert _{L^{2}\left(  0,L\right)  }%
^{2}dt\leq\frac{1}{n}\left\Vert u_{0,n}\right\Vert _{L^{2}\left(  0,L\right)
}^{2}\text{,} \label{m6}%
\end{equation}
where $u_{0,n}=u_{n}\left(  0\right)  $. Since $\alpha_{n}:=\left\Vert
u_{0,n}\right\Vert _{L^{2}\left(  0,L\right)  }\leq R_{0}$, one can choose a
subsequence of $\left\{  \alpha_{n}\right\}  $, still denoted by $\left\{
\alpha_{n}\right\}  $, such that%
\[
\lim\limits_{n\rightarrow\infty}\alpha_{n}=\alpha\text{.}%
\]
There are two possible cases: $i.$ $\alpha>0$ and $ii.$ $\alpha=0$.

\begin{itemize}
\item[i.] $\alpha>0$.
\end{itemize}

Note that the sequence $\left\{  u_{n}\right\}  $ is bounded in
$L^{\infty}\left(  0,T;L^{2}\left(  0,L\right)  \right)  \cap L^{2}\left(
0,T;H^{1}\left(  0,L\right)  \right)  $. On the other hand,%
\[
u_{n},_{t}=-\left(  u_{n,x}+\frac{1}{2}\partial_{x}\left(  u_{n}^{2}\right)
+u_{n,xxx}-Gu_{n}\right)  ,
\]
is bounded in $L^{2}\left(  0,T;H^{-2}\left(  0,L\right)  \right)  $. \ As the
first immersion of%
\[
H^{1}\left(  0,L\right)  \hookrightarrow L^{2}\left(  0,L\right)
\hookrightarrow H^{-2}\left(  0,L\right)  \text{,}%
\]
is compact, exists a subsequence, still denoted by $\left\{  u_{n}\right\}  $,
such that%
\begin{equation}%
\begin{array}
[c]{lll}%
u_{n}\longrightarrow u\text{ in }L^{2}\left(  0,T;L^{2}\left(  0,L\right)
\right)  \text{,}\\
-\frac{1}{2}\partial_{x}\left(  u_{n}^{2}\right) \rightharpoonup-\frac{1}%
{2}\partial_{x}\left(  u^{2}\right)\text{ in }L^{2}\left(  0,T;H^{-1}\left(
0,L\right)  \right)  \text{.}%
\end{array}
\label{m7}%
\end{equation}
It follows from (\ref{m6}) and (\ref{m7}) that%
\begin{equation}
\int_{0}^{T}\left\Vert Gu_{n}\right\Vert _{L^{2}\left(  0,L\right)  }%
^{2}dt\overset{n\rightarrow\infty}{\longrightarrow}\int_{0}^{T}\left\Vert
Gu\right\Vert _{L^{2}\left(  0,L\right)  }^{2}=0,\label{m10}%
\end{equation}
which implies that
\[
Gu=0,
\]
i.e.,
\[
u\left(  x,t\right)  -\frac{1}{\left\vert \omega\right\vert }\int_{\omega
}u\left(  x,t\right)  dx=0\Rightarrow u\left(  x,t\right)  =\frac
{1}{\left\vert \omega\right\vert }\int_{\omega}u\left(  x,t\right)  dx.
\]
Consequently,%
\[
u\left(  x,t\right)  =c\left(  t\right)  \text{ in }\omega\times\left(
0,T\right)  \text{,}%
\]
for some function $c\left(  t\right)  $. Thus, letting $n\rightarrow\infty$,
we obtain from (\ref{m1}) that%
\begin{equation}
\left\{
\begin{array}
[c]{lll}%
u_{t}+u_{x}+u_{xxx}=f,&&\text{ in }\left(  0,L\right)  \times\left(  0,T\right)
\text{,}\\
u=c\left(  t\right), && \text{ in }\omega\times\left(  0,T\right)  \text{.}%
\end{array}
\right.  \label{m11}%
\end{equation}
Let $w_{n}=u_{n}-u$ and $f_{n}=-\frac{1}{2}\partial_{x}\left(  u_{n}%
^{2}\right)  -f-Gu_{n}$. Note first that,
\begin{equation}
\int_{0}^{T}\left\Vert Gw_{n}\right\Vert _{L^{2}\left( 0,L\right)  }%
^{2}dt=\int_{0}^{T}\left\Vert Gu_{n}\right\Vert _{L^{2}\left( 0,L\right)
}^{2}dt+\int_{0}^{T}\left\Vert Gu\right\Vert _{L^{2}\left(  0,L\right)
}^{2}dt-2\int_{0}^{T}\left(  Gu_{n},Gu\right)  _{L^{2}\left(0,L\right)
}dt\rightarrow0\text{.}\label{m12}%
\end{equation}
Since $w_{n}\rightharpoonup0$ in $L^{2}\left(  0,T;H^{1}\left(  0,L\right)
\right)  $, we infer from Rellich's Theorem that $\int_{0}^{L}w_{n}\left(
y,t\right)  dy\rightarrow0$ strongly in $L^{2}\left(  0,T\right)  $. Combining (\ref{m7}) and (\ref{m12}), we have that
\[
\int_{0}^{T}\int_{0}^{L}\left\vert w_{n}\right\vert ^{2}\longrightarrow
0\text{.}%
\]
Thus,%
\[
w_{n,t}+w_{n,x}+w_{n,xxx}=f_{n}\text{,}%
\]%
\[
f_{n}\rightharpoonup0\text{ in }L^{2}\left(  0,T;H^{-1}\left(  0,L\right)
\right)  \text{ ,}%
\]
and,
\[
w_{n}\longrightarrow0\text{ in }L^{2}\left(  0,T;L^{2}\left(  0,L\right)
\right)  \text{,}%
\]
so,
\[
\partial_{x}\left(  w_{n}^{2}\right)  \longrightarrow w_{x}^{2}%
\]
in the sense of distributions. Therefore, $f=-\frac{1}{2}\partial_{x}\left(
u^{2}\right)  $ e $u\in L^{2}\left(  0,T;L^{2}\left(  0,L\right)  \right)  $
satisfies%
\[
\left\{
\begin{array}
[c]{lll}%
u_{t}+u_{x}+u_{xxx}+\frac{1}{2}\left(  u^{2}\right)  _{x}=0, &&\text{ in }\left(
0,L\right)  \times\left(  0,T\right),\\
u=c\left(  t\right),&&  \text{ in }\omega\times\left(  0,T\right).
\end{array}
\right.
\]
The first equation gives $c^{\prime}\left(  t\right)  =0$ which, combined with
unique continuation property (see Appendix \ref{Apendice}), yields that $u\left(
x,t\right)  =c$ for some constant $c\in\mathbb{R}$. Since $u(L,t)=0$, we
deduce that%
\[
0=u\left(  L,t\right)  =c\text{,}%
\]
and $u_{n}$ converges strongly to $0$ in $L^{2}\left(  0,T;L^{2}\left(
0,L\right)  \right)  $. We can pick some time $t_{0}\in\left[  0,T\right]  $
such that $u_{n}\left(  t_{0}\right)  $ tends to $0$ strongly in $L^{2}\left(
0,L\right)  $. Since%
\[
\left\Vert u_{n}\left(  0\right)  \right\Vert _{L^{2}\left(  0,L\right)  }%
^{2}\leq\left\Vert u_{n}\left(  t_{0}\right)  \right\Vert _{L^{2}\left(
0,L\right)  }^{2}+\int_{0}^{t_{0}}\left\Vert Gu_{n}\right\Vert _{L^{2}\left(
0,L\right)  }^{2}dt\text{,}%
\]
it is inferred that $\alpha_{n}=\left\Vert u_{n}\left(  0\right)  \right\Vert
_{L^{2}\left(  0,L\right)  }\longrightarrow0$, as $n\rightarrow\infty$, which
is in contradiction with the assumption $\alpha>0$.

\begin{itemize}
\item[ii.] $\alpha=0$.
\end{itemize}

First, note that $\alpha_{n}>0$, for all $n$. Set $v_{n}=u_{n}/\alpha_{n}$,
for all $n\geq1$. Then,%
\[
v_{n,t}+v_{n,x}+v_{n,xxx}-Gv_{n}+\frac{\alpha_{n}}{2}\left(  v_{n}^{2}\right)
_{x}=0
\]
and%
\begin{equation}
\int_{0}^{T}\left\Vert Gv_{n}\right\Vert _{L^{2}\left( 0,L\right)  }%
^{2}dt<\frac{1}{n}\text{.}\label{m13}%
\end{equation}
Since
\begin{equation}
\left\Vert v_{n}\left(  0\right)  \right\Vert _{L^{2}\left(  0,L\right)
}=1\text{,}\label{m14}%
\end{equation}
the sequence $\left\{  v_{n}\right\}  $ is bounded in $L^{2}\left(
0,T;L^{2}\left(  0,L\right)  \right) \cap L^{2}\left(  0,T;H^{1}\left(  0,L\right)  \right)  $, and,
therefore, $\left\{  \partial_{x}\left(  v_{n}^{2}\right)  \right\}  $ is
bounded in $L^{2}\left(  0,T;L^{2}\left(  0,L\right)  \right)  $. Then,
$\alpha_{n}\partial_{x}\left(  v_{n}^{2}\right)  $ tends to $0$ in this space.
Finally, $$\int_{0}^{T}\left\Vert Gv\right\Vert _{L^{2}\left(  0,L\right)
}^{2}dt=0.$$ Thus, $v$ is solution of
\[
\left\{
\begin{array}
[c]{lll}%
v_{t}+v_{x}+v_{xxx}=0,&&\text{ in }\left(  0,L\right)  \times\left(  0,T\right),\\
v=c\left(  t\right),&&  \text{ in }\omega\times\left(  0,T\right).
\end{array}
\right.
\]
We infer that $v\left(  x,t\right)  =c\left(  t\right)  =c$, thanks to Holmgren's
Theorem, and that $c=0$ due the fact that $v\left(  L,t\right)  =0$. 

According to the previous fact, pick a time $t_{0}\in\left[  0,T\right]  $ such that
$v_{n}\left(  t_{0}\right)  $ converges to $0$ strongly in $L^{2}\left(
0,L\right)  $. Since%
\[
\left\Vert v_{n}\left(  0\right)  \right\Vert _{L^{2}\left(  0,L\right)  }%
^{2}\leq\left\Vert v_{n}\left(  t_{0}\right)  \right\Vert _{L^{2}\left(
0,L\right)  }^{2}+\int_{0}^{t_{0}}\left\Vert Gv_{n}\right\Vert _{L^{2}\left(
0,L\right)  }^{2}dt\text{,}%
\]
we infer from (\ref{m13}) that $\left\Vert v_{n}\left(  0\right)
\right\Vert _{L^{2}\left(  0,L\right)  }\rightarrow0$, which contradicts to
(\ref{m14}). The proof is complete.
\end{proof}

\section{Comments and extensions for other models}\label{Sec4}
In this section we intend to analyze the results obtained in this manuscript as well as to present some extensions of these results for other models.  

\subsection{Comments of the results} In this work we deal with the KdV equation from a control point of
view posed in a bounded domain $(0,L)\subset\mathbb{R}$ with a \textit{forcing term} $Gh$ added as a control input, namely:%
\begin{equation}
\left\{
\begin{array}
[c]{lll}%
u_{t}+u_{x}+uu_{x}+u_{xxx}+Gh=0,&&\text{in }\left(  0,L\right)
\times\left(  0,T\right),\\
u\left(  0,t\right)  =u\left(  L,t\right)  =u_{x}\left(  L,t\right)
=0,&&\text{in }\left(  0,T\right),\\
u\left(x,0\right)  =u_0\left(x\right),&&  \text{in }\left(
0,L\right)  .
\end{array}
\right.  \label{FC}%
\end{equation}
Here $G$ is the operator defined by \eqref{I11a}.

The result presented in this manuscript gives us a new \textit{"weak" forcing mechanism} that ensures  global stability to the  system \eqref{FC}. In fact, Theorem \ref{main1} guarantees a lower cost to control the system proposed in this work and, consequently, to derive a good result related with the stabilization problem as compared with existing results in the literature. 

The interested readers can look at the following article \cite{Pazoto}, related to what we call \textit{"strong" forcing mechanism}. Indeed,  in this article, the author proposed the source term as $1_{\omega}h(x,t)$, that is,  the mechanism proposed does not remove a medium term as seen in $Gh$ defined by \eqref{I11a}. 

Finally, observe that the approach used to prove our main result as well as the weak mechanism can be extended for \textit{KdV-type equation} and for \textit{a model of strong interaction between internal solitary waves.} Let us breviary describe these systems and the results that can be derived  by using the same approach applied in this work.

\subsection{KdV-type equation} Fifth-order KdV type equation can be written as
\be\label{kaw}
u_t+u_x+\beta u_{xxx}+ \alpha u_{xxxxx}+uu_x=0,
\ee
where $u=u(t,x)$ is a real-valued function of two real variables $t$ and $x$, $\alpha$ and $\beta$ are real constants. When we consider, in \eqref{kaw},  $\beta=1$ and $\alpha=-1$, T. Kawahara  \cite{Kawahara} introduced a dispersive partial differential equation which describes one-dimensional propagation of small-amplitude long waves in various problems of fluid dynamics and plasma physics, the so-called Kawahara equation. 

With the damping mechanism proposed in this manuscript, we can investigate the stabilization problem, already mentioned in this article, for the following system
\begin{equation}
\left\{
\begin{array}
[c]{lll}%
u_{t}+u_{x}+uu_{x}+u_{xxx}-u_{xxxxx}+Gh=0,  &  & \text{in }(  0,T)  \times(  0,L)  \text{,}\\
u(  t,0)  =u(  t,L)  =u_{x}(  t,0)=u_{x}(  t,L)  =u_{xx}(  t,L)=0, &  & \text{in }(  0,T)  \text{,}\\
u(  0,x)  =u_{0}(  x)  &  & \text{in }( 0,L), 
\end{array}
\right.  \label{S1}%
\end{equation}
and $G$ as in  \eqref{I11}.

In fact, a similar result can be obtained with respect to global stabilization. Obviously, we need to pay attention to the unique continuation property for this case (for our case see Appendix \ref{Apendice}). However, due the Carleman estimate provided by Chen in \cite{MoChen}, it is possible to show the unique continuation property for the Kawahara operator.

\subsection{Model of strong interaction between internal solitary waves} We can consider a model of two KdV equations types. Precisely, in \cite{geargrinshaw}, a complex system of equations was derived by Gear and Grimshaw to model the strong interaction of two-dimensional, long, internal gravity waves propagating on neighboring pycnoclines in a stratified fluid. It has the structure of a pair of Korteweg-de Vries equations coupled through both dispersive and nonlinear effects and has been the object of intensive research in recent years. In particular, we also refer to \cite{bona} for an extensive discussion on the physical relevance of the system.

An interesting possibility now presents itself is the study of the stability properties when the model is posed on a bounded domain $(0,L)$, that is, to study the Gear-Grimshaw system with only a weak damping mechanism, namely,
\begin{equation}
\label{gg1}
\begin{cases}
u_t + uu_x+u_{xxx} + a_3v_{xxx} + a_1vv_x+a_2 (uv)_x =0, & \text{in} \,\, (0,L)\times (0,\infty),\\
c v_t +rv_x +vv_x+a_3b_2u_{xxx} +v_{xxx}+a_2b_2uu_x+a_1b_2(uv)_x  +Gv=0,  & \text{in} \,\, (0,L)\times (0,\infty), \\
u(x,0)= u^0(x), \quad v(x,0)=  v^0(x), & \text{in} \,\, (0,L),
\end{cases}
\end{equation}
satisfying the following boundary conditions
\begin{equation}\label{gg2}
\begin{cases}
u(0,t)=0,\,\,u(L,t)=0,\,\,u_{x}(L,t)=0, & \text{in} \,\, (0,\infty),\\
v(0,t)=0,\,\,v(L,t)=0,\,\,v_{x}(L,t)=0, & \text{in} \,\, (0,\infty),
\end{cases}
\end{equation}
where $a_1, a_2, a_3, b_2, c, r$ are constants in  $\mathbb{R}$ assuming physical relations. Here, as in all work, $Gv$ is the weak forcing term defined in \eqref{I11}.

The stabilization problem for the system \eqref{gg1}-\eqref{gg2} was addressed in \cite{capistrano}. The author showed that the total energy associated with the model decay exponentially when $t$ tends to $\infty$, considering two damping mechanisms $Gu$ and $Gv$ acting in both equations of \eqref{gg1}. However, even though the  system \eqref{gg1} has the structure of a pair of KdV equations, it cannot be decoupled into two single KdV equations\footnote{Remark that the uncoupling is not possible in \eqref{gg1} unless $r = 0$.} and, in this case, the result shown in this work is not a consequence of the results proved in \cite{capistrano}.
 
Lastly, B\'arcena--Petisco \textit{et al.} in a recent work \cite{BaGuePa}, addressed the controllability problem for the system \eqref{gg2}, by means of a control $1_{\omega}f(x,t)$, supported in an interior open subset of the domain and acting on one equation only. The proof consists mainly on proving the controllability of the linearized system, which is done by getting a Carleman estimate for the adjoint system.  With this result in hand, by using $Gv$ as a control mechanism, instead of $1_{\omega}f(x,t)$, it is possible to prove the global stabilization for the model \eqref{gg2}. As in the KdV (see Appendix \ref{Apendice}) and Kawahara cases, we need to prove a unique continuation property to achieve the stabilization problem, however with the Carleman estimate \cite[Proposition 3.2]{BaGuePa}, we are able to derive this property for the Gear--Grimshaw operator.

\subsection{About exact controllability results}
Now, we will discuss the exact controllability property of the KdV system
\begin{equation}
\left\{
\begin{array}
[c]{lll}%
u_{t}+u_{x}+uu_{x}+u_{xxx}=Gw,&&\text{in }\left(  0,L\right)
\times\left(  0,T\right)\text{,}\\
u\left(  0,t\right)  =u\left(  L,t\right)  =u_{x}\left(  L,t\right)
=0, &&\text{in }(0,T)\text{,}\\
u\left(  x,0\right)  =u_{0}\left(  x\right), && \text{in } \left(
0,L\right)  .
\end{array}
\right.  \label{h1}%
\end{equation}
with weak source term $G$ defined by
\begin{equation*}
Gw\left(  x,t\right)  =1_{\omega}\left(  w\left(  x,t\right)  -\frac
{1}{\left\vert \omega\right\vert }\int_{\omega}w\left(  x,t\right)  dx\right),
\end{equation*}
where $\omega\subset(0,L) $ and
$1_{\omega}$ denotes the characteristic function of the set $\omega$. We arises in the following open question:

\vspace{0.2cm}

\noindent\textbf{Control problem:} \textit{Given an initial state $u_{0}$ and a terminal state $u_{1}$ in $L^2(0,L)$, can one find an appropriate control input $w\in L^2(\omega\times(0,T))$ so that the equation
(\ref{h1}) admits a solution $u$ which satisfies $u\left(  \cdot,0\right)
=u_{0}$ and $u\left(  \cdot,T\right)  =u_{1}$?}

\vspace{0.2cm}

It is important to point out that we do not expect that system \eqref{h1} has the  exact control property as above mentioned when we consider the control $w$ in $L^2(\omega\times(0,T))$. Roughly speaking, (large) negative waves propagate from the right to the left. Therefore, a negative wave cannot be generated by a left control, that means, when the control is acting far from the endpoint $x = L $, i.e. in some interval $\omega=(l_1,l_2)$ with $0<l_1<l_2<L$, then there is no chance to control exactly the state function on $(l_2, L)$,  (see e.g. \cite{Rosier2}). However, we believe that using the techniques proposed in \cite{CaRoPa} (or in \cite{GG}), i. e., considering the weight Sobolev spaces (or control more regular), there is a chance to get positive answer for exact control problem in the right hand side of the domain, precisely, considering $\omega=(L-\epsilon,L)$, with the weak control as defined in \eqref{h2}.

\subsection{A natural damping mechanism}When we consider the boundary condition of \eqref{FC} with $G=0$, a natural feedback law is revealed as we can see below 
\begin{equation}\label{energyFC}
\frac{\mathrm{d} E}{\mathrm{d} t}=-\frac{1}{2}\left|u_{x}(0, t)\right|^{2}
\end{equation}
with
$$
E(t)=\frac{1}{2} \int_{0}^{L}|u(x, t)|^{2} \mathrm{d} x.
$$
The energy dissipation law \eqref{energyFC} shows that the boundary value problem under consideration is dissipated through the extreme $x =0$ and leads one to guess that any solution of \eqref{FC}, with $G=0$,  may decay to zero as $t\to\infty$. In order to answer this question, a really nonlinear method is needed, and the method applied here can not be addressed to solve it.

\appendix
\section{Unique continuation property}\label{Apendice}
This appendix aims to provide a sketch of how to obtain the unique continuation property through a Carleman estimate.
\subsection{Carleman inequality}
Pick any function $\psi \in C^3([0,L])$ with 
\ba
&&\psi >0\textrm{ in } [0,L], \label{C1} \quad|\psi '|>0, \quad \psi ''<0,\quad \textrm{ and } \quad\psi '\psi ''' <0 \textrm{ in } [0,L],\\
&&\psi '(0)<0 , \quad \psi '(L) >0, \quad\textrm{ and } \quad
\max_{x\in [0,L]}\psi (x)= \psi (0)=\psi (L).  \label{C4}
\ea

Set
\be\label{psi}
\vf (t,x) =\frac{\psi (x) }{t(T-t)} \cdot
\ee
For $f\in L^2(0,T; L^2(0,L))$ and $q_0\in L^2(0,L)$, let $q$ denote the solution of the system
\begin{equation}
\left\{
\begin{array}
[c]{lll}%
q_t + q_x+q_{xxx} =f,&&t\in (0,T),\  x\in (0,L) ,\\
q(t,0)=q(t,L)=q_x(t,L)=0 && t\in (0,T),\\
q(0,x) =q_0(x), &&  \text{in }\left(
0,L\right).
\end{array}
\right.  \label{A1}
\end{equation}
Thus, the following result is a direct consequence of the Carleman estimate proved by \cite{CaRoPa}.
\begin{proposition}
\label{prop10} Pick any $T>0$. 
There exist two constants $C>0$ and $s_0 >0$ such that any $f\in L^2(0,T;L^2(0,L))$, any $q_0\in L^2(0,L)$ and any 
$s\ge s_0$,
the solution $q$ of
\eqref{A1} fulfills
\begin{equation}
\int_0^T\!\! \int_0^L [s\vf |q_{xx}|^2 +(s\vf )^3 |q_x|^2 +(s\vf )^5 |q|^2 ]e^{-2s\vf} dx dt\le C\left( \int_0^T\!\!\int_0^L |f|^2 e^{-2s\vf} dxdt \right),
\label{C7}
\end{equation}
where $\varphi$ is defined by \eqref{A1} and $\psi$  satisfies \eqref{C1}--\eqref{C4}.
\end{proposition}
Actually,  Proposition \ref{prop10} will play a great role in establishing the unique continuation property describes below. 

\begin{corollary}\label{UCP1}
Let $L>0$ and $T>0$ be two real numbers, and let $\omega\subset\left(
0,L\right)  $ be a nonempty open set. If $v\in L^{\infty}\left(
0,T;H^{1}\left(  0,L\right)  \right)  $ solves%
\[
\left\{
\begin{array}
[c]{lll}%
v_{t}+v_{x}+v_{xxx}+vv_{x}=0, &  & \text{in }\left(  0,L\right)  \times\left(
0,T\right),\\
v\left(  0,t\right)  =0, &  & \text{in }\left(
0,T\right),\\
v=c, &  & \text{in } (l',L)\times\left(  0,T\right),
\end{array}
\right.
\]
with $0<l'<L$ and $c\in\mathbb{R}$, then $v\equiv c$ in $\left(  0,L\right)  \times\left(  0,T\right)  $.
\end{corollary}

\begin{proof} We do not expect that $v$ belongs to $$
L^{2}\left(0, T ; H^{3}(0, l)\right) \cap H^{1}\left(0, T ; L^{2}(0, l)\right).
$$ In this way, we have to smooth it. For any function $v=v(x, t)$ and any $h>0$, let us consider $v^{[h]}(x, t)$ defined by
\[
v^{[h]}(x, t):=\frac{1}{h} \int_{t}^{t+h} v(x, s) ds.
\]
Remember that if $v \in L^{p}(0, T ; V),$ where $1 \leq p \leq+\infty$ and $V$ denotes any Banach space, we have that $$v^{[h]} \in W^{1, p}(0, T-h ; V)$$
$$\left\|v^{[h]}\right\|_{L^{p}(0, T-h ; V)} \leq\|v\|_{L^{p}(0, T ; V)},$$
and
\[
v^{[h]} \rightarrow v \quad \text { in } L^{p}\left(0, T^{\prime} ; V\right) \text { as } h \rightarrow 0,
\]
for $p<\infty$ and $T^{\prime}<T$. 

Choose any $T^{\prime}<T .$ Thus, for a small enough number $h$, $$v^{[h]} \in W^{1, \infty}\left(0, T^{\prime} ; H_{0}^{1}(0, l)\right)$$ and $v^{[h]}$ is solution of 
\begin{equation}\label{3.4}
v_t^{[h]}+v_x^{[h]}+ v_{xxx}^{[h]}+\left(vv_x\right)^{[h]}=0 \quad \text{ in } (0, l) \times\left(0, T^{\prime}\right),
\end{equation}
\begin{equation}\label{3.5}
v^{[h]}(0, t)=0 \quad \text{ in } \left(0, T^{\prime}\right)
\end{equation}
and
\begin{equation}\label{3.6}
v^{[h]} \equiv c \quad \text{ in }\left(l^{\prime}, l\right) \times\left(0, T^{\prime}\right),
\end{equation}
for some $c\in\mathbb{R}$. Since $v \in L^{\infty}\left(0, T ; H^{1}(0, l)\right)$ and $vv_x\in L^{\infty}\left(0, T ; L^{2}(0, l)\right)$, therefore, it follows from \eqref{3.4}, that $$v_{x x x}^{[h]} \in L^{\infty}\left(0, T^{\prime} ; L^{2}(0, l)\right)$$ and thus $$v^{[h]} \in L^{\infty}\left(0, T^{\prime} ; H^{3}(0, l)\right) .$$ 
Thanks to the Carleman estimate \eqref{C7}, we get that
\begin{eqnarray}
\int_0^{T^{\prime}}\!\! \int_0^L [s\vf |v^{[h]}_{xx}|^2 +(s\vf )^3 |v^{[h]}_x|^2 +(s\vf )^5 |v^{[h]}|^2 ]e^{-2s\vf} dx dt&\le &C\left( \int_0^{T^{\prime}}\!\!\int_0^L |f|^2 e^{-2s\vf} dxdt \right)\nonumber\\
&\leq &2 C_{0} \int_{0}^{T^{\prime}} \int_{0}^{l}\left|vv_x^{[h]}\right|^{2} e^{-2s\vf}d x d t \label{3.7}\\
&+&2 C_{0} \int_{0}^{T^{\prime}} \int_{0}^{l}\left|\left(vv_x\right)^{[h]}-vv_x^{[h]}\right|^{2} e^{-2s\vf} d x d t\nonumber \\
&:=&I_{1}+I_{2}\nonumber,
\end{eqnarray}
for any $s \geq s_{0}$ and $\vf (t,x) $ defined by \eqref{psi}.

\vspace{0.2cm}

\noindent \textbf{Claim 1:} \textit{$I_{1}$ is bounded and can be absorbed by the left-hand side of \eqref{3.7}.}

\vspace{0.2cm}

In fact, since $v \in L^{\infty}\left(0, T ; L^{\infty}(0, l)\right),$ we have
\begin{equation}\label{3.8}
I_{1} \leq C \int_{0}^{T^{\prime}} \int_{0}^{l}\left|v_x^{[h]}\right|^{2} e^{-2s\vf} d x d t,
\end{equation}
for some constant $C>0$ which does not depend on $h$. Comparing the powers of $s$ in the right-hand side of \eqref{3.8} with those in the left-hand side of \eqref{3.7} we deduce that the $\operatorname{term} I_{1}$ in \eqref{3.7} may be dropped by increasing the constants $C_{0}$ and $s_{0}$ in a convenient way, getting Claim 1.

\vspace{0.2cm}

\noindent \textbf{Claim 2:} \textit{$I_{2} \rightarrow 0$, as $h \rightarrow 0 .$}

\vspace{0.2cm}

From now on,  fix $s$, which means, to the value $s_{0}$. Thanks to the fact that $e^{-2s_0\vf} \leq 1,$ it is sufficient to prove that
\begin{equation}\label{3.9}
\left(vv_x\right)^{[h]} \rightarrow vv_x \quad \text{ in } L^{2}\left(0, T^{\prime} ; L^{2}(0, l)\right)
\end{equation}
and
\begin{equation}\label{3.10}
vv_x^{[h]} \rightarrow vv_x \quad \text{ in }L^{2}\left(0, T^{\prime} ; L^{2}(0, l)\right).
\end{equation}

In fact, since $$vv_x \in L^{2}\left(0, T' ; L^{2}(0, l)\right)$$ \eqref{3.9} holds and, from the fact that $v\in L^{\infty}\left(0, T' ; L^{\infty}(0, l)\right)\cap L^{2}\left(0, T' ; H^{1}(0, l)\right)$, \eqref{3.10} follows, showing the Claim 2.

By Claims 1 and 2, as $h \rightarrow 0$, the integral term
\[
\int_0^{T^{\prime}}\!\! \int_0^L [s\vf |v^{[h]}_{xx}|^2 +(s\vf )^3 |v^{[h]}_x|^2 +(s\vf )^5 |v^{[h]}|^2 ]e^{-2s\vf} dx dt\to 0.
\]
On the other hand, $v^{[h]} \rightarrow v$ in $L^{2}\left(0, T^{\prime} ; L^{2}(0, l)\right)$. It follows that $v \equiv c$ in $(0, l) \times\left(0, T^{\prime}\right)$, for $c\in\mathbb{R}$. As $T^{\prime}$ may be taken arbitrarily close to $T,$ we infer that $v \equiv c$ in $(0, l) \times(0, T)$, for some $c\in\mathbb{R}$. This completes the proof of Corollary \ref{UCP1}.
\end{proof}

As a consequence of Corollary \ref{UCP1}, we give below the \textit{unique continuation property.}

\begin{corollary}
Let $L>0$, $T>0$ be real numbers, and $\omega\subset\left(
0,L\right)  $ be a nonempty open set. If $v\in L^{\infty}\left(
0,T;H^{1}\left(  0,L\right)  \right)  $ is solution of
\[
\left\{
\begin{array}
[c]{lll}%
v_{t}+v_{x}+v_{xxx}+vv_{x}=0, &  & \text{in }\left(  0,L\right)  \times\left(
0,T\right),\\
v\left(  0,t\right)  =v\left( L,t\right)  =0, &  & \text{in }\left(
0,T\right),\\
v=c, &  & \text{in } \omega\times\left(  0,T\right),
\end{array}
\right.
\]
where $c\in\mathbb{R}$, then $v\equiv c$ in $\left(  0,L\right)  \times\left(  0,T\right)  $.
\end{corollary}

\begin{proof}
Without loss of generality we may assume that $\omega=\left(l_{1}, l_{2}\right)$ with $0 \leq$ $l_{1}<l_{2} \leq L .$ Pick $l=\left(l_{1}+l_{2}\right) / 2 .$ First, apply  Corollary \ref{UCP1} to the function $v(x, t)$ on $(0, l) \times(0, T)$. After that, we use the following change of variable $v(L-x, T-t)$ on $(0, L-l) \times(0, T),$ to conclude that $v \equiv c$ on $(0, L) \times(0, T)$, achieving the result.
\end{proof}

\vspace{0.2cm}

\noindent\textbf{Acknowledgement:}
The author wish to thank the referee for his/her valuable comments which improved this paper. 

\vspace{0.2cm}

Capistrano–Filho was supported by CNPq 306475/2017-0, 408181/2018-4, CAPES- PRINT 88881.311964/2018-01, CAPES-MATHAMSUD 88881.520205/2020-01, MATHAMSUD 21- MATH-03 and Propesqi (UFPE).


\end{document}